\documentclass[11pt]{article}

\usepackage{amsmath}
\usepackage{amssymb}
\usepackage{amsfonts}
\usepackage{amsthm}
\usepackage{array}
\usepackage{bm}
\usepackage{enumitem}
\usepackage{stmaryrd}

\usepackage{amsxtra}
\usepackage{array}
\usepackage{mathrsfs}
\usepackage{slashed}
\usepackage{xcolor}
\usepackage{tikz-cd}

\newcolumntype{C}[1]{>{\centering\hspace{0pt}}p{#1}}
\usepackage[margin=1in]{geometry}

\newtheorem{thm}{Theorem}[section]
\newtheorem{prop}[thm]{Proposition}
\newtheorem{lem}[thm]{Lemma}
\newtheorem{cor}[thm]{Corollary}

\theoremstyle{definition}
\newtheorem{defn}[thm]{Definition}

\theoremstyle{definition}
\newtheorem*{rmk}{Remark}

\usepackage[nottoc]{tocbibind}
\numberwithin{equation}{section}

\usepackage[hidelinks]{hyperref}

\title{Free-Boundary Problems for \\ Holomorphic Curves in the $6$-Sphere}
\author{Jesse Madnick}
\date{June 2021}

\newcommand{\Addresses}
{{   \bigskip
\noindent \textsc{National Center for Theoretical Sciences} \par\nopagebreak
\noindent   \textsc{National Taiwan University}\par\nopagebreak
\noindent   \textsc{Taipei, Taiwan}\par\nopagebreak
\noindent  \textit{E-mail address}: \texttt{jmadnick@ncts.ntu.edu.tw}
}}

\begin{document}

\maketitle

\begin{abstract}
We remark on two different free-boundary problems for holomorphic curves in nearly-K\"{a}hler $6$-manifolds.  First, we observe that a holomorphic curve in a geodesic ball $B$ of the round $6$-sphere that meets $\partial B$ orthogonally must be totally geodesic.  Consequently, we obtain rigidity results for reflection-invariant holomorphic curves in $\mathbb{S}^6$ and associative cones in $\mathbb{R}^7$. \\
\indent Second, we consider holomorphic curves with boundary on a Lagrangian submanifold in a strict nearly-K\"{a}hler $6$-manifold.  By deriving a suitable second variation formula for area, we observe a topological lower bound on the Morse index.  In both settings, our methods are complex-geometric, closely following arguments of Fraser-Schoen and Chen-Fraser.
%By deriving a second variation formula for area, together with Riemann-Roch, we derive
%In this note, we remark on two different free-boundary problems for holomorphic curves in the round $6$-sphere.  First, we observe that a holomorphic curve in a geodesic ball of $\mathbb{S}^6$ that meets the boundary sphere orthogonally must be totally geodesic.  As a consequence, we obtain a rigidity result for reflection-invariant associative cones in $\mathbb{R}^7$.  \\
%\indent Second, we consider holomorphic curves in a nearly-K\"{a}hler $6$-manifold with boundary on a Lagrangian submanifold, and derive a lower bound for the Morse index.  In both settings, our methods are complex-geometric, closely following arguments of Fraser-Schoen and Chen-Fraser.  % Our methods in this work are not original, entailing only minor variations on arguments of Fraser-Schoen and Chen-Fraser.
\end{abstract}

% \tableofcontents
% \pagebreak

\section{Introduction}

% \subsection{Background: Minimal Surfaces in Spheres}

\indent \indent The $6$-sphere is the only $n$-sphere for $n \geq 3$ that admits an almost complex structure.  Viewing $\mathbb{S}^6$ as the unit sphere in $\mathbb{R}^7$, the \textit{standard almost-complex structure} $J$ at $p \in \mathbb{S}^6$ is given by $J_p(v) = p \times v$, where $\times$ is the usual the $7$-dimensional cross product on the imaginary octonions $\text{Im}(\mathbb{O}) = \mathbb{R}^7$.  While $J$ is not integrable, it does act by isometries with respect to the round metric and is invariant under the compact Lie group $\text{G}_2$.  An immersed surface $u \colon \Sigma^2 \to \mathbb{S}^6$ is a \textit{holomorphic curve} (or \textit{complex curve}) if each tangent space is $J$-invariant:
$$J(T_p\Sigma) = T_p\Sigma, \ \ \ \forall p \in \Sigma.$$
%Its standard almost-complex structure $J$ arises from viewing $\mathbb{S}^6$ as the unit sphere in $\mathbb{R}^7 = \text{Im}(\mathbb{O})$ the imaginary octonions.  From this perspective, $J$ can be seen to be essentially the so-called $7$-dimensional vector cross product.  Although this $J$ is non-integrable, it acts by isometries with respect to the round metric. ...... With respect to this $J$. there is a distinguished class of surfaces.  Precisely, an immersed surface $u \colon \Sigma^2 \to \mathbb{S}^6$ is a \textit{holomorphic curve} (or \textit{complex curve}) if each tangent space is $J$-invariant:
% $$J(T_p\Sigma) = T_p\Sigma, \ \ \forall p \in \Sigma.$$
\indent Holomorphic curves in $\mathbb{S}^6$ are minimal surfaces.  One way to see this is via $\text{G}_2$ geometry.  Namely, a surface $\Sigma \subset \mathbb{S}^6$ is a holomorphic curve if and only if its ($3$-dimensional) cone $\text{C}(\Sigma) := \{rp \in \mathbb{R}^7 \colon r > 0, p \in \Sigma\} \subset\mathbb{R}^7$ is an \textit{associative $3$-fold}, one of the four classes of calibrated geometries discovered in the pioneering work of Harvey and Lawson \cite{harvey1982calibrated}.  Associative $3$-folds are fundamental objects in $\text{G}_2$ geometry, and holomorphic curves in $\mathbb{S}^6$ model their conical singularities.  In sum, if $\Sigma \subset \mathbb{S}^6$ is a holomorphic curve, then $\text{C}(\Sigma) \subset \mathbb{R}^7$ is an associative $3$-fold, so $\text{C}(\Sigma)$ is homologically volume-minimizing, and thus $\Sigma$ is a minimal surface. \\
\indent There are by now very many studies of holomorphic curves in the $6$-sphere (see the surveys \cite[$\S$19]{chen00riemannian} and \cite[$\S$12.1-$\S$12.3]{joyce2007}), though research in this direction is still ongoing (see, e.g., the recent works \cite{martins2013superminimal}, \cite{fernandez15space}, \cite{madnick2021second}).  Thus far, global questions have primarily concerned \textit{closed} surfaces.  In this paper, by contrast, we study holomorphic curves with boundary. \\
\indent Now, the last decade has seen great advances in the subject of minimal surfaces with boundary, particularly in the context of the free-boundary condition.  A minimal surface $u \colon \Sigma^2 \to B$ in a domain $B$ is \textit{free-boundary} if $u(\partial \Sigma) \subset \partial B$ and $u(\Sigma)$ intersects $\partial B$ orthogonally.  The orthogonality requirement here arises naturally from the first variation of area.  We will not attempt to survey this rapidly developing field of study --- an excellent recent overview is \cite{li2019free} --- but simply mention a particular result of interest.

\pagebreak

\indent In 2015, Fraser and Schoen \cite{fraser2015uniqueness} proved a remarkable rigidity theorem.  They showed that %if $\Sigma \approx D$ is homeomorphic to a $2$-disk, then
any free-boundary minimal $2$-disk $u \colon D^2 \to B$ in a geodesic ball $B$ of a real space form must be totally geodesic.  Thus, in geodesic balls $B$ of round spheres $\mathbb{S}^n$, any free-boundary minimal surface must have non-trivial topology.  In the case of $\mathbb{S}^6$, it is natural to ask what happens if the topological assumption ($\Sigma$ is a $2$-disk) is replaced by a geometric one (viz., $u(\Sigma)$ holomorphic).  In $\S3$, we show:

\begin{thm} \label{thm:Classical} Let $u \colon \Sigma^2 \to B$ be a compact immersed surface in a geodesic ball $B$ of the round $6$-sphere with $u(\partial \Sigma) \subset \partial B$.  If $u$ is a holomorphic curve, and if $u(\Sigma)$ meets $\partial B$ orthogonally, then $u(\Sigma)$ is totally geodesic.
\end{thm}

\indent Note that Theorem \ref{thm:Classical} makes no topological assumptions about $\Sigma$; that $\Sigma$, if connected, is homeomorphic to a $2$-disk is part of the conclusion.  Moreover, Theorem \ref{thm:Classical} immediately yields a rigidity result for closed, connected holomorphic curves $\Sigma$ in $\mathbb{S}^6 \subset \mathbb{R}^7$ that are reflection-invariant across a hyperplane, say $P = \{x_7 = 0\} \subset \mathbb{R}^7$.  Indeed, if such $\Sigma$ meets $P$ transversely, then $\Sigma \cap \{x_7 \geq 0\}$ is free-boundary in the upper hemisphere, and hence is totally-geodesic.  That is: % Indeed, such a hyperplane partitions $\mathbb{S}^6$ into the upper hemisphere $S^+$, lower hemisphere $S^-$, and equator $\mathbb{S}^6 \cap P$. .....

\begin{cor} \label{thm:ReflectHolo} Let $u \colon \Sigma^2 \to \mathbb{S}^6$ be a holomorphic curve, where $\Sigma^2$ is a closed connected surface, and let $P$ denote a totally geodesic $5$-sphere in $\mathbb{S}^6$.  If $u(\Sigma)$ is invariant under reflection in $P$, and if the intersection of $u(\Sigma)$ with $P$ is transverse, then $u(\Sigma)$ is a totally geodesic $2$-sphere.
\end{cor}

% Since holomorphic curves in $\mathbb{S}^6$ are precisely the links of associative cones in $\mathbb{R}^7$, we deduce: % (\textcolor{red}{check this carefully})

\begin{cor} \label{thm:ReflectAssoc} Let $f \colon N^3 \to \mathbb{R}^7$ be an associative cone whose link in $\mathbb{S}^6$ is a closed connected surface, and let $P \subset \mathbb{R}^7$ be a $6$-plane.  If $f(N)$ is invariant under reflection in $P$, and if the intersection of $f(N)$ with $P$ is transverse, then $f(N)$ is a $3$-plane.
\end{cor}

\indent We now turn to the Lagrangian free-boundary problem.  For holomorphic curves in symplectic manifolds, this is a well-studied boundary condition with a wealth of applications \cite{gromov1985pseudo}, \cite{floer1988morse}, \cite{oh1993floer}.  While the $6$-sphere is not symplectic, the pair $(\langle \cdot, \cdot \rangle, J)$ consisting of the round metric and standard almost-complex structure is, in fact, the prototype of a ``strictly nearly-K\"{a}hler" structure. \\ % Let us quickly recall this terminology. \\%it is quite interesting to study the Lagrangian boundary problem in that setting.  Let us quickly recall this terminology.
\indent In general, recall that an almost-Hermitian $6$-manifold $(M^6, \langle \cdot, \cdot \rangle, J, \omega)$, consisting of a metric $\langle \cdot, \cdot \rangle$, compatible almost-complex structure $J$, and $2$-form $\omega(X,Y) := \langle JX, Y \rangle$ is \textit{nearly-K\"{a}hler} if
$$(\nabla_XJ)(Y) = -(\nabla_YJ)(X), \ \ \forall X,Y \in \Gamma(TM),$$
where $\nabla$ is the Levi-Civita connection.  By definition, a K\"{a}hler $6$-manifold ($\nabla J = 0$) is nearly-K\"{a}hler.  A nearly-K\"{a}hler $6$-manifold is \textit{strict} if it is not K\"{a}hler.  Strict nearly-K\"{a}hler $6$-manifolds are of great importance to $\text{G}_2$-geometry as models of $\text{G}_2$-holonomy cones \cite{bar1993real}. \\
\indent In $\S$4, we study compact holomorphic curves $u \colon \Sigma \to M^6$ with boundary $u(\partial \Sigma) \subset L$ on a fixed Lagrangian $3$-fold $L \subset M$ in a nearly-K\"{a}hler $6$-manifold $M$.  For normal variations $\eta \in \Gamma(N\Sigma)$ with $\eta$ tangent to $L$ along $\partial \Sigma$, we prove the second variation of area formula:
\begin{equation}
(\delta^2A)(\eta) = \int_\Sigma \frac{1}{2} \Vert \mathscr{D}\eta \Vert^2 + \frac{1}{3} d\omega(e, J\eta, \mathscr{D}_{e}\eta)  - 2\lambda^2 \Vert \eta \Vert^2. \label{eq:SecondVarFirst}
\end{equation}
Here, $\mathscr{D}$ is the $J$-antilinear part of the normal connection, $e \in \Gamma(T\Sigma)$ is any (local) unit vector field, and $\lambda \geq 0$ is a constant (the \textit{type} of $M$) having $\lambda > 0$ if and only if $M^6$ is strict.  A remarkable feature of this formula is that it contains \textit{no boundary integral}, a consequence of the trinity of conditions ``$M$ nearly-K\"{a}hler," ``$u$ holomorphic," and ``$L$ Lagrangian."  Using (\ref{eq:SecondVarFirst}), together with a suitable index formula, we deduce the following topological lower bound on the Morse index:

\begin{thm} \label{thm:LagIndexBound} Let $u \colon \Sigma \to M^6$ be a holomorphic curve in a strictly nearly-K\"{a}hler $6$-manifold, where $\Sigma$ is compact orientable surface with boundary, such that $u(\partial \Sigma) \subset L$ for some Lagrangian submanifold $L \subset M$.  Then the Morse index of $u$ for the area functional satisfies
\begin{equation}
\mathrm{Ind}(u) \geq \mu(u^*(TM), TL), \label{eq:IndexBound}
\end{equation}
where $\mu(u^*(TM), TL) \in \mathbb{Z}$ is the boundary Maslov index of the bundle pair $(u^*(TM), TL) \to (\Sigma, \partial \Sigma)$.
\end{thm}

\begin{rmk} In the edge case of $\partial \Sigma = \O$, the bound (\ref{eq:IndexBound}) gives no information, as that situation has $\mu(u^*(TM), \O) = 2 c_1(u^*(TM)) = 0$.  More generally, as the boundary Maslov index is \textit{integer}-valued, the bound (\ref{eq:IndexBound}) is vacuous whenever $\mu(u^*(TM), TL) \leq 0$.  It is not clear to the author how often this occurs.  In fact, while there are by now plenty of examples of holomorphic curves and Lagrangian submanifolds (see \cite[$\S$19]{chen00riemannian}, \cite[$\S$12]{joyce2007}, \cite{lotay2011ruled}, \cite{storm2020lagrangian} and references therein) in strict nearly-K\"{a}hler $6$-manifolds, there presently seem to be few explicit examples of pairs $(u(\Sigma), L)$ of holomorphic curves $u(\Sigma)$ with Lagrangian boundary $L$.  Constructing such pairs is an interesting direction for future research.
\end{rmk}

\begin{rmk} Recently, Pacini \cite{pacini2019maslov} discovered a beautiful Chern-Weil type integral formula for the boundary Maslov index.  This formula may yield more precise information on the bound (\ref{eq:IndexBound}) in various special cases (e.g., if $M^6 = \mathbb{S}^6$, say).  We leave this to the interested reader. % It is possible that using the Chern-Weil formula of Pacini [cite], one ought be able to get more precise information in special cases (e.g., if the ambient $6$-manifold is the $6$-sphere, say).  We leave this to the interested reader.
\end{rmk}

\indent This work is organized as follows.  In $\S$2, we establish conventions and fix notation.  We also review the basics of nearly-K\"{a}hler manifolds and holomorphic curves for readers unfamiliar with these subjects.  In $\S$3, we study holomorphic curves in geodesic balls of $\mathbb{S}^6$.  We prove Theorem \ref{thm:Classical} by a short Hopf differential argument analogous to that of Fraser--Schoen \cite{fraser2015uniqueness}. \\
\indent In $\S$4 we study holomorphic curves with Lagrangian boundary.  In $\S$4.1, we derive the second variation formula (\ref{eq:SecondVarFirst}), and in $\S$4.3 we deduce Theorem \ref{thm:LagIndexBound} by appealing to a suitable generalization of the Riemann-Roch formula.  Our arguments follow those of Chen--Fraser \cite{chen2010holomorphic} who study holomorphic curves with Lagrangian boundary in (K\"{a}hler) complex projective spaces.  Note that sections $3$ and $4$ can be read independently of one another, although both rely on $\S$2.  \\
%we derive the second variation of area  % This section primarily serves to establish our conventions and fix notation.  In $\S$3, we ..... Our proof of Theorem X is a short Hopf differential argument, closely following that of Fraser and Schoen \cite{fraser2015uniqueness}. \\
%\indent Finally, in $\S$4, we.....  Again, our argument is complex-geometric in spirit, appealing a suitable generalization of the Riemann-Roch formula.  Here, our arguments follow those of Chen and Fraser \cite{chen2010holomorphic} who study holomorphic curves with Lagrangian boundary in (K\"{a}hler) complex projective spaces.  Note that sections $3$ and $4$ can be read independently of one another, although both rely on $\S$2. \\

\noindent \textbf{Acknowledgements:} I thank Benjamin Aslan, Jonny Evans, and Chung-Jun Tsai for helpful conversations related to this work.  I also thank Gavin Ball, Da Rong Cheng, Spiro Karigainnis, Wei-Bo Su, and Albert Wood for their interest and encouragement.  This work was completed during the author's postdoctoral fellowship at the National Center for Theoretical Sciences (NCTS) at National Taiwan University; I thank the Center for their support.

% \pagebreak

\section{Preliminaries}

\subsection{Nearly-K\"{a}hler $2n$-Manifolds}

\indent \indent Let $M^{2n}$ be an \textit{almost-Hermitian} $2n$-manifold of real dimension $2n \geq 6$, meaning that $M^{2n}$ is equipped with a triple $(\langle \cdot, \cdot \rangle, J, \omega)$ consisting of a Riemannian metric $\langle \cdot, \cdot \rangle$, an orthogonal almost-complex structure $J$, and a non-degenerate $2$-form $\omega \in \Omega^2(M)$ given by $\omega(X,Y) = \langle JX, Y \rangle$.  We say that $M^{2n}$ is \textit{K\"{a}hler} if any of the following equivalent conditions are satisfied:
\begin{align*}
\nabla J = 0 \iff \nabla \omega & = 0 \iff J \text{ integrable and }\omega \text{ closed},
\end{align*}
where $\nabla$ is the Levi-Civita connection of $\langle \cdot, \cdot \rangle$.  The obstruction to $M^{2n}$ being K\"{a}hler is measured by the \textit{intrinsic torsion} tensor $P = \nabla J \in \Gamma(T^*M \otimes T^*M \otimes TM)$ given by
$$P(X,Y) = (\nabla_XJ)(Y).$$
Note that $P$ satisfies the symmetries
\begin{align}
P(X,JY) & = -J P(X,Y) & \langle P(X,Y), Z \rangle & = -\langle P(X,Z), Y \rangle.
\label{eq:TorsSym}
\end{align}
Note also that $P = \nabla J$ and $\nabla \omega$ carry essentially equivalent information, in view of the fact that
$$(\nabla_X\omega)(Y,Z) = \langle (\nabla_XJ)(Y), Z \rangle = \langle P(X,Y), Z \rangle.$$
\indent Fix $m \in M$, abbreviate $T = T_mM$, and note that the standard $\text{U}(n)$-representation on $T \simeq \mathbb{R}^{2n}$ induces a $\text{U}(n)$-representation on $T^* \otimes T^* \otimes T$, and hence on the $\text{U}(n)$-invariant subspace $E \subset T^* \otimes T^* \otimes T$ consisting of tensors with the symmetries (\ref{eq:TorsSym}).  In a now classical paper \cite{MR581924}, Gray and Hervella showed (for $n \geq 3$) that $E$ decomposes into four irreducible $\text{U}(n)$-submodules
$$E \cong V_1 \oplus V_2 \oplus V_3 \oplus V_4$$
of real dimensions
\begin{align*}
\frac{1}{3}n(n-1)(n-2), \ \ \ \ \ \frac{2}{3}n(n-1)(n+1), \ \ \ \ \ n(n+1)(n-2), \ \ \ \ \ 2n
\end{align*}
respectively.  %Here, we reiterate that our standing assumption is $n \geq 3$.
%\begin{align*}
%\dim(V_1) & = \frac{1}{3}n(n-1)(n-2) & \dim(V_2) & = \frac{2}{3}n(n-1)(n+1)  \\
%\dim(V_3) & = n(n+1)(n-2) & \dim(V_4) & = 2n.
%\end{align*}

Letting $\mathcal{V}_1, \ldots, \mathcal{V}_4 \subset T^*M \otimes T^*M \otimes TM$ denote the vector subbundles over $M$ corresponding to the subspaces $V_1, \ldots, V_4 \subset T^* \otimes T^* \otimes T$, respectively, Gray and Hervella further showed that:
\begin{align*}
\nabla J \in \Gamma(\mathcal{V}_1) & \iff (\nabla_XJ)(Y) = -(\nabla_YJ)(X),  \ \ \forall X,Y \in TM \\
\nabla J \in \Gamma(\mathcal{V}_2) & \iff d\omega = 0 \\
\nabla J \in \Gamma(\mathcal{V}_3 \oplus \mathcal{V}_4) & \iff J \text{ integrable}.
\end{align*}
We say that $M$ is \textit{nearly-K\"{a}hler} if $\nabla J \in \Gamma(\mathcal{V}_1)$, i.e.:
$$P(X,Y) = -P(Y,X), \ \ \forall X,Y \in TM.$$
An equivalent condition \cite{MR581924} is that $\nabla \omega = \frac{1}{3}d\omega$.   In particular, every K\"{a}hler manifold is nearly-K\"{a}hler.  More precisely, the K\"{a}hler manifolds are exactly those nearly-K\"{a}hler manifolds with $J$ integrable (or, equivalently, with $d\omega = 0$). \\
%  Note that a nearly-K\"{a}hler manifold has $J$ integrable if and only if $\omega$ is closed, which in turn happens if and only if it is K\"{a}hler. \\
 %  In particular, every K\"{a}hler manifold is nearly-K\"{a}hler. \\  % In real dimensions $2$ and $4$, the converse is also true, so the nearly-K\"{a}hler condition is only significant in real dimension $2n \geq 6$. \\

% \indent A nearly K\"{a}hler manifold is said to be \textit{strict} if $\nabla_XJ \neq 0$ for all non-zero $X \in TM$.  In particular, every strict nearly-K\"{a}hler manifold is not K\"{a}hler.  In real dimension $6$, Gray showed \cite[Theorem 5.2]{gray1976structure} the converse: every nearly-K\"{a}hler $6$-manifold that is not K\"{a}hler must, in fact, be strict. \\

\indent A nearly-K\"{a}hler manifold is said to be \textit{strict} if $\nabla_XJ \neq 0$ for all non-zero $X \in TM$.  In particular, every strict nearly-K\"{a}hler manifold is not K\"{a}hler.  In real dimension $6$, Gray showed \cite[Theorem 5.2]{gray1976structure} that the converse holds: If $M^6$ is nearly-K\"{a}hler and not K\"{a}hler, then there exists a constant $\lambda > 0$, called the \textit{type}, for which %every nearly-K\"{a}hler $6$-manifold that is not K\"{a}hler must, in fact, be strict. \\
\begin{equation}
\Vert (\nabla_XJ)(Y) \Vert^2 = \lambda^2 \left[ \Vert X \Vert^2 \Vert Y \Vert^2 - \langle X,Y \rangle^2 - \langle JX, Y \rangle^2 \right]
\label{eq:ConstantType}
\end{equation}
for all $X,Y \in TM$, and hence $M$ is strict.  Of course, if $M^6$ is K\"{a}hler, then equation (\ref{eq:ConstantType}) is still true with $\lambda = 0$. \\

\indent If $M^{2n}$ is a nearly-K\"{a}hler $2n$-manifold, its Riemann curvature tensor $R(X,Y)Z = [\nabla_X,\nabla_Y]Z - \nabla_{[X,Y]}Z$ enjoys several symmetries with respect to the almost-complex structure $J$.  A convenient list of identities is given in \cite[$\S$2]{gray1976structure}, though we caution that Gray uses the opposite sign convention for $R$.  In this work, we will only need the identity
%\indent The Riemann curvature tensor $\overline{R}$ of a nearly-K\"{a}hler $2n$-manifold enjoys several additional symmetries beyond the usual ones.  A convenient list of these relations is given in [cite], though we caution that Gray's sign convention for $\overline{R}$ differs from ours.  In this work, we will only need the following
%\indent The nearly-K\"{a}hler condition implies that the Riemann curvature tensor $R$ satisfies several additional symmetries beyond the usual ones.  These symmetries have been studied quite thoroughly by Gray \cite{gray1976structure}.  For example, Gray showed (possibly this was shown earlier?) that the curvature tensor of a nearly-K\"{a}hler $2n$-manifold satisfies
\begin{equation*}
R(X,Y,Z,X) + R(X,JY,JZ,X) + R(X, JX, Y, JZ) = 2 \langle P(X,Y), P(X,Z) \rangle
% \label{eq:BasicCurvId}
\end{equation*}
given in \cite[(2.4)]{gray1976structure}.  In particular, setting $Y = Z$ yields the relation
\begin{equation}
R(X,Y,Y,X) + R(X,JY,JY,X) + R(X,JX,Y,JY) = 2\Vert P(X,Y) \Vert^2.
\label{eq:KeyCurvId}
\end{equation}

%\begin{equation}
%R(X,Y,Y,X) + R(X,Y,JX,JY) = \Vert P(X,Y) \Vert^2, \ \ \ \forall X,Y \in TM.
%\label{eq:BasicCurvId}
%\end{equation}
%More precisely, we will need the following consequence:
%%This identity has the following consequefAnce, which we will use later:
%
%\begin{prop} Let $M^{2n}$ be a nearly-K\"{a}hler $2n$-manifold.  Then
%\begin{equation}
%R(X,Y,Y,X) + R(X,JY,JY,X) + R(X,JX,Y,JY) = 2\Vert P(X,Y) \Vert^2 \\
%\label{eq:KeyCurvId}
%\end{equation}
%\end{prop}
%
%\begin{proof} Replacing $Y$ with $JY$ in (\ref{eq:BasicCurvId}) yields
%\begin{align}
%R(X,JY,JY,X) - R(X,JY,JX,Y) & = \Vert P(X,Y) \Vert^2 \label{eq:YwithJY}
%\end{align}
%so adding (\ref{eq:BasicCurvId}) and (\ref{eq:YwithJY}) gives
%\begin{equation}
%R(X,Y,Y,X) + R(X,JY,JY,X) + R(X,Y,JX,JY) - R(X,JY,JX,Y) = 2\Vert P(X,Y) \Vert^2. 
%\label{eq:CurvFourTerms}
%\end{equation}
%Separately, the algebraic Bianchi implies
%%$$R(X,W,Y,Z) = R(X,Y,W,Z) - R(X,Z,W,Y).$$
%%In particular, taking $Z = JY$ and $W = JX$ gives
%$$R(X,JX,Y,JY) = R(X,Y,JX,JY) - R(X,JY,JX,Y).$$
%Plugging this equation into (\ref{eq:CurvFourTerms}) proves (\ref{eq:KeyCurvId}).
%\end{proof}

%\begin{prop} Curvature identities
%\end{prop}

\subsection{Nearly-K\"{a}hler $6$-Manifolds and $\text{G}_2$ Geometry}

% \indent \indent For the remainder of this work, we focus on the special case of nearly-K\"{a}hler $6$-manifolds.
%\indent \indent We now specialize to the case of nearly-K\"{a}hler $6$-manifolds. \\

\indent \indent Among nearly-K\"{a}hler manifolds, the strict nearly-K\"{a}hler 6-manifolds $M^6$ are particularly special.  In that case, the $3$-form $\frac{1}{3\lambda}\,d\omega$ gives rise to a complex volume form on $M^6$, which in turn leads to a relationship with $\text{G}_2$-holonomy cones.  The purpose of this section is to explain this point.

\subsubsection{Elements of $\text{G}_2$ Geometry}

\indent \indent Let $X^7$ be a smooth $7$-manifold.  A \textit{$\mathrm{G}_2$-structure} on $X^7$ is a $3$-form $\phi \in \Omega^3(X)$ such that at each $x \in X$, the symmetric bilinear form $B_\phi \colon \text{Sym}^2(T_xX) \to \Lambda^7(T_x^*X)$ given by
$$\textstyle B_\phi(v,w) := \frac{1}{6}\,(v\,\lrcorner\,\phi) \wedge (w\,\lrcorner\,\phi) \wedge \phi$$
is definite.  It is well-known that $X$ admits a $\text{G}_2$-structure if and only if $X$ is orientable and spin.  Moreover, a $\text{G}_2$-structure induces a Riemannian metric $g_\phi$ and orientation $\text{vol}_\phi$ according to
\begin{align*}
g_\phi(v,w)\,\text{vol}_\phi & = \textstyle \frac{1}{6}\,(v\,\lrcorner\,\phi) \wedge (w\,\lrcorner\,\phi) \wedge \phi \\
\text{vol}_\phi & = \phi \wedge \ast\phi.
\end{align*}
A $\text{G}_2$-structure $\phi$ is called \textit{torsion-free} if it is closed and co-closed --- i.e., if $d\phi = 0$ and $d(\ast \phi) = 0$. The relationship between torsion-free $\text{G}_2$-structures and Riemannian metrics with $\text{G}_2$ holonomy is given by the following theorem of Fern\'{a}ndez and Gray:
%A theorem of Fern\'{a}ndez and Gray asserts that

\begin{thm}[\cite{fernandez1982riemannian}] Let $X^7$ be orientable and spin.  If $\phi \in \Omega^3(X)$ is a torsion-free $\mathrm{G}_2$-structure on $X$, then its induced metric $g_\phi$ has $\mathrm{Hol}(g_\phi) \leq \mathrm{G}_2$.  Conversely, if $g$ is a Riemannian metric on $X$ with $\mathrm{Hol}(g) \leq \mathrm{G}_2$, then $g = g_\phi$ for some torsion-free ${G}_2$-structure $\phi$ on $X$.
\end{thm}

%\begin{thm}[cite] Let $Y^7$ be orientable and spin.  If a $\text{G}_2$-structure $\phi \in \Omega^3(Y)$ satisfies $d\phi = 0$ and $d\ast\!\phi = 0$, then its induced metric $g_\phi$ has $\text{Hol}(g_\phi) \leq \text{G}_2$.  Conversely, if $g$ is a Riemannian metric on $Y$ with $\text{Hol}(g) \leq \text{G}_2$, then $g = g_\phi$ for some $\text{G}_2$-structure $\phi$ that satisfies $d\phi = 0$ and $d\ast\!\phi = 0$.
%\end{thm}

\indent Let $X$ have a $\text{G}_2$-structure $\phi \in \Omega^3(M)$.  A $3$-dimensional submanifold $N^3 \subset X^7$ is called \textit{associative} if
$$\left.\phi\right|_N = \text{vol}_N.$$
If $d\phi = 0$, then $\phi$ is a calibration and its calibrated submanifolds are precisely the associative $3$-folds.  Thus, in this case, associatives are homologically volume-minimizing.  Associative $3$-folds are fundamental objects in $\text{G}_2$-geometry, and are in many ways analogous to holomorphic curves in symplectic geometry.

%\indent Note also that if $\phi$ is a closed $\text{G}_2$-structure, then $\phi$ is a calibration.  The $\phi$-calibrated submanifolds are known as ``associative $3$-folds."  That is, a $3$-dimensional submanifold $N^3 \subset Y^7$ is \textit{associative} if
%$$\phi|_N = \text{vol}_N.$$
%Associative $3$-folds are fundamental objects in $\text{G}_2$-geometry.

%\subsubsection{Complex Volume Forms on Strict Nearly-K\"{a}hler $6$-Manifolds}

\subsubsection{Strict Nearly-K\"{a}hler $6$-Manifolds and $\text{G}_2$ Holonomy Cones}

\indent \indent On an almost-Hermitian manifold $(M^{2n}, \langle \cdot, \cdot \rangle, J, \omega)$, a \textit{complex volume form} is an $(n,0)$-form $\Upsilon$ satisfying the normalization
$$\textstyle (-1)^{n(n-1)/2}\left( \frac{i}{2}\right)^n \Upsilon \wedge \overline{\Upsilon} = \text{vol}_M.$$
%An almost-Hermitian structure $(\langle \cdot, \cdot \rangle, J, \omega)$ together with a choice of complex volume form $\Upsilon$ is called an \textit{$\text{SU}(n)$-structure}.
% A quadruple $(\langle \cdot, \cdot \rangle, J, \omega, \Upsilon)$ is called an \textit{$\text{SU}(n)$-structure}. % Note that $\Upsilon$ itself is a frame for the complex line bundle $\Lambda^{n,0}(M) \to M$.  Hence, if a complex volume form exists, then $\Lambda^{n,0}(M)$ is trivializable.
If $(M^6, J, \omega)$ is a strict nearly-K\"{a}hler $6$-manifold, then it is well-known \cite[$\S$4.3]{reyes1993some} that each of the $3$-forms (for a constant $\theta \in [0,2\pi)$) %for each constant $\theta \in [0,2\pi)$, the $3$-form
$$\Upsilon_\theta(X,Y,Z) := \frac{e^{i\theta}}{3\lambda} \left( d\omega(X,Y,Z) - i\, d\omega(X,Y,JZ) \right)$$
is a complex volume on $M$.  Moreover, the pair $(\omega, \Upsilon_\theta)$ satisfies the differential equations
\begin{align}
d\omega & = 3\lambda \left( \cos(\theta)\,\text{Re}(\Upsilon_\theta) + \sin(\theta)\,\text{Im}(\Upsilon_\theta) \right) \notag \\
d\,\text{Re}(\Upsilon_\theta) & = 2\lambda\sin(\theta)\,\omega \wedge \omega \label{eq:NKEqns} \\
d\,\text{Im}(\Upsilon_\theta) & = -2\lambda\cos(\theta)\,\omega \wedge \omega. \notag
\end{align}

\begin{rmk} In $\S$2.1, we defined ``nearly-K\"{a}hler" as a particular class of almost-Hermitian structures $(\langle \cdot, \cdot\rangle, J, \omega)$.  This is the original definition used by Gray \cite{gray1965minimal}.  However, in real dimension $6$, it has recently become common for authors to use the term ``nearly-K\"{a}hler $6$-manifold" to mean a ``strict nearly-K\"{a}hler $6$-manifold together with a fixed $\Upsilon_\theta$ (usually $\Upsilon_0$ or $\Upsilon_{\pi/2}$) and scaled to have $\lambda = 1$."
% Although we have defined a nearly-K\"{a}hler structure as a type of almost-Hermitian structure $(\langle \cdot, \cdot \rangle, J, \omega)$, many authors regard the choice of a particular $\Upsilon_\theta$ (typically $\Upsilon_0$ or $\Upsilon_{\pi/2}$) as part of the defining data of nearly-K\"{a}hler structure. 
\end{rmk}
% For more on this point of view, see [X thesis] or [X thesis]. \\
%\begin{rmk} Additional data as $\text{SU}(3)$-structure.  We will need this when we discuss $\mathbb{S}^6$.
%\end{rmk}

% \subsubsection{Strict Nearly-K\"{a}hler $6$-Manifolds and $\text{G}_2$ Holonomy Cones}

%\indent \indent Nearly-K\"{a}hler $6$-manifolds are of basic importance to $\text{G}_2$-geometry by way of the following theorem of Bar (cite).  The statement we are giving here is adapted from (Morris thesis).

%\begin{thm} Relation to $\text{G}_2$-holonomy cones.  Consequently, Einstein.
%\end{thm}

\indent For a Riemannian manifold $(M,g)$, recall that its \textit{metric cone} is the Riemannian manifold
$$\text{C}(M) := (\mathbb{R}^+ \times M, dr^2 + r^2g).$$
We will routinely identify $M$ with the subset $\{1\} \times M \subset \text{C}(M)$ and call $M$ the \textit{link} of $\text{C}(M)$.  The relationship between strict nearly-K\"{a}hler $6$-manifolds and $\text{G}_2$ geometry is given by: 

\begin{thm}[\cite{bar1993real}\label{thm:Bar}] If $M^6$ is a strict nearly-K\"{a}hler $6$-manifold, scaled to have $\lambda = 1$, then $\mathrm{C}(M)$ admits a torsion-free $\mathrm{G}_2$-structure.  Conversely, if $\mathrm{C}(M)$ is a $7$-dimensional metric cone with a torsion-free $\mathrm{G}_2$-structure, then its link $M^6$ admits a strict nearly-K\"{a}hler structure with $\lambda = 1$.
\end{thm}

\indent Since manifolds with holonomy contained in $\text{G}_2$ are Ricci-flat, and since the links of Ricci-flat cones are Einstein of positive scalar curvature, it follows that:

\begin{cor}[\cite{gray1976structure}] Every strict nearly-K\"{a}hler $6$-manifold is Einstein of positive scalar curvature.
\end{cor}

\indent Here is a sketch of Theorem \ref{thm:Bar}. If $(M^6, \langle \cdot, \cdot \rangle, J, \omega)$ is strictly nearly-K\"{a}hler with $\lambda = 1$, then define a $3$-form $\phi \in \Omega^3(\text{C}(M))$ and $4$-form $\psi \in \Omega^4(\text{C}(M))$ via
\begin{align*}
\phi & := r^2\,dr \wedge \omega + r^3\,\text{Re}(\Upsilon_0) & \psi & := \textstyle -r^3\,dr \wedge \text{Im}(\Upsilon_0) + \frac{1}{2}r^4\,\omega \wedge \omega.
\end{align*}
Linear algebra shows that $\phi$ is a $\text{G}_2$-structure on $\text{C}(M)$ and that $\psi = \ast \phi$.  Equations (\ref{eq:NKEqns}) for $(\omega, \Upsilon_0)$ imply that $\phi = d\!\left(\frac{r^3}{3}\omega\right)$ and $\psi = d\!\left(-\frac{r^4}{4}\,\text{Im}(\Upsilon_0)\right)$, so $\phi$ is torsion-free.

Conversely, if $\phi \in \Omega^3(\text{C}(M))$ is a torsion-free $\text{G}_2$-structure, then one can define a $2$-form $\omega \in \Omega^2(M)$ and a complex $3$-form $\Upsilon_0 \in \Omega^3(M; \mathbb{C})$ via
\begin{align*}
\omega & := \left.(\partial_r \,\lrcorner\,\phi)\right|_M & \Upsilon_0 & := \left.(\phi - i\,\partial_r \lrcorner \ast\!\phi)\right|_M\!.
\end{align*}
Linear algebra shows that $\omega$ is a non-degenerate $2$-form that is $\langle \cdot, \cdot \rangle$-compatible, so that defining $J$ via $\langle X, JY \rangle = \omega(X,Y)$ makes $(\langle \cdot, \cdot \rangle, J, \omega)$ an almost-Hermitian structure on $M$.  Moreover, $\Upsilon_0$ is a complex volume form with respect to $J$.  The torsion-free condition can then be used to calculate $d\omega = 3\,\text{Re}(\Upsilon_0)$, from which one can deduce that $(\langle \cdot, \cdot \rangle, J, \omega)$ is strictly nearly-K\"{a}hler with $\lambda = 1$.  This concludes the sketch.  For more details, see \cite[$\S$7]{bar1993real}, \cite{morris2014nearly}, \cite[$\S$2]{van2019lagrangian}.

\subsection{Holomorphic Curves in Nearly-K\"{a}hler $6$-Manifolds}

\indent \indent We now define our primary objects of interest.  In an almost-Hermitian manifold $(M^{2n}, \langle \cdot, \cdot \rangle, J, \omega)$, a \textit{holomorphic curve} is an immersion $u \colon \Sigma^2 \to M^{2n}$ such that each tangent space is $J$-invariant:
$$J(T_p\Sigma) = T_p\Sigma, \ \ \forall p \in \Sigma.$$
An equivalent condition is that
$$\left.\omega\right|_\Sigma = \text{vol}_\Sigma$$
where $\text{vol}_\Sigma$ is the volume form on $\Sigma$.  If $d\omega = 0$, then $\omega$ is a calibration, and hence holomorphic curves are homologically area-minimizing. 

\indent Now, if $M^6$ is a strict nearly-K\"{a}hler $6$-manifold, then $\omega$ is not closed, and hence not a calibration.  Nevertheless, the holomorphic curves in $M^6$ are rather special geometric objects, as the following well-known fact shows: %are minimal surfaces.  In fact, they are precisely the links of a distinguished class of ($3$-dimensional) calibrated cones in $C(M)$:

% \indent [Definition].  As mentioned in the introduction, holomorphic curves in nearly-K\"{a}hler $6$-manifolds have been studied by .....

\begin{prop} Let $M^6$ be a strict nearly-K\"{a}hler $6$-manifold.  Let $\Sigma^2 \subset M^6$ be an immersed oriented surface.  Then $\Sigma^2 \subset M^6$ is a holomorphic curve if and only if $\mathrm{C}(\Sigma) \subset \mathrm{C}(M)$ is an associative $3$-fold.
%Let $\Sigma \subset M^6$ Holomorphic curve iff cone is associative.  Consequently, every holomorphic curve is a minimal surface.
\end{prop}

Consequently, if $\Sigma^2 \subset M^6$ is a holomorphic curve, then $\text{C}(\Sigma) \subset \text{C}(M)$ is a ($3$-dimensional) calibrated submanifold, and hence $\Sigma$ is a minimal surface in $M$. \\

% \indent To conclude this section, we recall some basic aspects of holomorphic curves. \\

\indent We now set our conventions for immersed submanifolds $u \colon \Sigma^k \to M^n$ of a Riemannian manifold $(M^n, \langle \cdot, \cdot \rangle)$.  In general, we let $\overline{\nabla}$ denote the Levi-Civita connection on $M$, and split $u^*(TM) = T\Sigma \oplus N\Sigma$ into tangential and normal parts.  For $X,Y \in \Gamma(T\Sigma)$ and $\eta \in \Gamma(N\Sigma)$, we have the decompositions
\begin{align*}
\overline{\nabla}_XY & = \nabla^\top_XY + \text{I\!I}(X,Y) \\
\overline{\nabla}_X\eta & = W_X\eta + \nabla^\perp_X\eta
\end{align*}
where $\nabla^\top$ is the Levi-Civita connection on $\Sigma$, where $\nabla^\perp$ is the normal connection, where $\text{I\!I}$ is the second fundamental form, and where $W$ is the shape operator satisfying the Weingarten equation
$$\langle W_XN, Y \rangle = -\langle \text{I\!I}(X,Y), N \rangle.$$
The curvature tensors of $\overline{\nabla}, \nabla^\top, \nabla^\perp$ will be denoted by $\overline{R}, R^\top, R^\perp$, respectively.  In particular, we recall the \textit{Ricci equation}
\begin{equation}
\overline{R}(X,Y, \eta, \xi) = R^\perp(X,Y, \eta, \xi) + \langle W_X\eta, W_Y\xi \rangle - \langle W_X\xi, W_Y\eta\rangle
\label{eq:RicciEq}
\end{equation}
for $X,Y \in T\Sigma$ and $\eta, \xi \in N\Sigma$.

\indent Suppose now that $u \colon \Sigma^2 \to M^6$ is an immersed holomorphic curve in a nearly-K\"{a}hler $6$-manifold $(M^6, \langle \cdot, \cdot \rangle, J, \omega)$.  In this case, it is easy to check that the second fundamental form enjoys the symmetries
\begin{align*}
 \text{I\!I}(X,JY) & = \text{I\!I}(JX,Y) & \text{I\!I}(X,JY) & = J\,\text{I\!I}(X,Y).
\end{align*}
When written in terms of the shape operator, these symmetries take the equivalent form
\begin{align}
 W_{JX}\eta & = -JW_X\eta & W_X(J\eta) & = J(W_X\eta) \label{eq:SymShapeOp}
\end{align}
for $X \in T\Sigma$ and $\eta \in N\Sigma$.  Note that the identity $\text{I\!I}(X,JY) = \text{I\!I}(JX,Y)$ is true more generally for any ($2$-dimensional) oriented minimal surface in a Riemannian manifold, where $J$ is the complex structure on $T\Sigma$ induced from the metric and orientation.  \\
\indent Finally, we remark for future use that the restriction of $P \colon TM \times TM \to TM$ to $T\Sigma \times N\Sigma$ maps into $N\Sigma$.  Indeed, by virtue of the symmetry $W_X(J\eta) = J(W_X\eta)$, the restricted map
$$P \colon T\Sigma \times N\Sigma \to N\Sigma$$
satisfies
$$P(X,\eta) = (\overline{\nabla}_XJ)(\eta) = (\nabla^\perp_XJ)(\eta)$$
for $X \in T\Sigma$ and $\eta \in N\Sigma$.

\section{Free-Boundary Holomorphic Curves in $\mathbb{S}^6$}

\indent \indent In this short section, we recall the standard strict nearly-K\"{a}hler structure on the $6$-sphere, set up the moving frame for holomorphic curves in $\mathbb{S}^6$, and prove Theorem 1.1.  Our argument relies on various holomorphic bundle structures that first appeared in \cite{bryant82}, though our notation follows \cite{madnick2021second}.

\subsection{The Round $6$-Sphere}

\indent \indent To begin, let us view $\mathbb{R}^7 = \text{Im}(\mathbb{O})$ as the imaginary octonions, and let $g_0$ denote the standard euclidean inner product on $\mathbb{R}^7$.  The so-called ``$7$-dimensional cross product" is the alternating bilinear map $\times \colon \text{Im}(\mathbb{O}) \times \text{Im}(\mathbb{O}) \to \text{Im}(\mathbb{O})$ given by
$$x \times y := \textstyle \frac{1}{2}(xy - yx).$$
Using the metric $g_0$ to lower an index, we can recast $\times$ as a tensor $\phi_0 \colon \text{Im}(\mathbb{O}) \times \text{Im}(\mathbb{O}) \times \text{Im}(\mathbb{O}) \to \mathbb{R}$ via
$$\phi_0(x,y,z) := g_0(x \times y, z).$$
The properties of octonionic multiplication show that $\phi_0$ is alternating.  In fact, $\phi_0 \in \Omega^3(\mathbb{R}^7)$ is the flat $\text{G}_2$-structure on $\mathbb{R}^7$.  It is clear that $\phi_0$ is torsion-free and that $\text{Hol}(g_0) = \{\text{Id}\}$. \\

\indent Let $(\mathbb{S}^6, \langle \cdot, \cdot \rangle)$ denote the round $6$-sphere with constant curvature $1$, and let us embed $\mathbb{S}^6 \subset \mathbb{R}^7$ in the standard way.  For each $p \in \mathbb{S}^6$, the linear map $J_p \colon T_p\mathbb{S}^6 \to T_p\mathbb{S}^6$ given by
\begin{align*}
J_p(x) & = p \times x
\end{align*}
satisfies $(J_p)^2 = -\text{Id}$ and $\langle J_px, J_py \rangle = \langle x,y \rangle$.  The map $J \colon T\mathbb{S}^6 \to T\mathbb{S}^6$ is the standard ($\text{G}_2$-invariant, orthogonal) almost-complex structure on the round $\mathbb{S}^6$.  Defining $\omega \in \Omega^2(\mathbb{S}^6)$ via
$$\omega(x,y) := \langle Jx, y\rangle$$
the triple $(\langle \cdot, \cdot \rangle, J, \omega)$ is an almost-Hermitian structure on $\mathbb{S}^6$.  In fact, $(\langle \cdot, \cdot \rangle, J, \omega)$ is a $\text{G}_2$-invariant, strict nearly-K\"{a}hler structure with $\lambda = 1$.  Moreover, the $3$-form $\Upsilon \in \Omega^{3,0}(\mathbb{S}^6)$
$$\Upsilon := \left.(\partial_r\,\lrcorner\,(\ast \phi) + i\phi)\right|_{\mathbb{S}^6}$$
is a complex volume form that satisfies $d\omega = 3\,\text{Im}(\Upsilon)$ and $d\,\text{Re}(\Upsilon) = 2\,\omega \wedge \omega$. \\
%\begin{align*}
%d\omega & = 3\,\text{Im}(\Upsilon) & d\,\text{Re}(\Upsilon) & = 2\,\omega \wedge \omega.
%d\,\text{Im}(\Upsilon) & = 0.
%\end{align*}
%One often says that the quadruple $(\langle \cdot, \cdot \rangle), J, \omega, \Upsilon)$ is an $\text{SU}(3)$-structure, as it allows for a notion of a 
%Let $J$ ..... Let $\omega$ ..... Then $(\langle \cdot, \cdot \rangle, J, \omega)$ is a $\text{G}_2$-invariant, strict nearly-K\"{a}hler structure with $\lambda = 1$.  Moreover, the $3$-form
%$$\Upsilon := \left.(\partial_r\,\lrcorner\,(\ast \phi) + i\phi)\right|_{\mathbb{S}^6}$$
%is a complex volume form ..... \\
% \indent \indent Explain $\text{SU}(3)$-structure. \\
\indent Let $\overline{\nabla}$ denote the Levi-Civita connection on $\mathbb{S}^6$.  Since $(\langle \cdot, \cdot \rangle, J, \omega)$ is strictly nearly-K\"{a}hler, we have both $\overline{\nabla} J \neq 0$ and $\overline{\nabla} \omega \neq 0$.
% $$\overline{\nabla} J \neq 0 \ \ \ \text{and} \ \ \ \ \overline{\nabla} \omega \neq 0.$$
%Therefore, for many geometric applications, it is more convenient
Thus, as far as the nearly-K\"{a}hler structure of $\mathbb{S}^6$ is concerned, the Levi-Civita connection $\overline{\nabla}$ is perhaps not the most geometrically natural one.  Indeed, in many applications, it is more convenient to use the \textit{nearly-K\"{a}hler connection} $\overline{D}$ on $T\mathbb{S}^6$ defined by
\begin{align*}
\overline{D}_XY & := \overline{\nabla}_XY + \frac{1}{2}P(X,JY),
\end{align*}
where we recall $P(X,Y) = (\overline{\nabla}_XJ)(Y)$. The advantage of $\overline{D}$ is that both $J$ and $\omega$ are $\overline{D}$-parallel.

\indent For computations, we extend both $\overline{\nabla}$ and $\overline{D}$ to the complexified tangent bundle $T\mathbb{S}^6 \otimes_{\mathbb{R}} \mathbb{C}$ by $\mathbb{C}$-linearity.  As usual, we decompose $T\mathbb{S}^6 \otimes_{\mathbb{R}} \mathbb{C}$ into types with respect to $J$:
$$T\mathbb{S}^6 \otimes_{\mathbb{R}} \mathbb{C} = T^{1,0}\mathbb{S}^6 \oplus T^{0,1}\mathbb{S}^6.$$
Since $\overline{D}$ preserves $J$, its restriction to $T^{1,0}\mathbb{S}^6$ is a well-defined connection, which we continue to denote $\overline{D}$.  We caution that the restriction of $\overline{\nabla}$ to $T^{1,0}\mathbb{S}^6$, by contrast, does not yield a well-defined connection.

% the more geometrically natural

%For many calculations on $\mathbb{S}^6$, the Levi-Civita conection Explain NK connection $\overline{D}$.

\subsection{Holomorphic Curves in $\mathbb{S}^6$}

%\indent \indent Let $u \colon \Sigma^2 \to \mathbb{S}^6$ be a holomorphic curve in the $6$-sphere, where ``holomorphic" is meant with respect to the standard strict nearly-K\"{a}hler structure on $\mathbb{S}^6$ discussed above.  Throughout this work, we will routinely abuse notation by using the same symbol for a connection on a vector bundle $E \to \mathbb{S}^6$ and its pullback connection on $u^*E \to \Sigma$.  In particular, we let $\overline{D}$ denote the pullback connection on $u^*(T^{1,0}\mathbb{S}^6) \to \Sigma$. \\  %We will use the same symbol $\overline{\nabla}$ to refer to the Levi-Civita connection on $T\mathbb{S}^6 \to \mathbb{S}^6$ as well as to its pullback on $u^*(T\mathbb{S}^6) \to \Sigma$.  The same abuse of notation will be applied to the nearly-K\"{a}hler connection $\overline{D}$. \\

%\indent Consider the complexified tangent bundle $T\mathbb{S}^6 \otimes \mathbb{C}$ and split it into its $(1,0)$ and $(0,1)$ parts with respect to $J$:
%$$T\mathbb{S}^6 \otimes \mathbb{C} = T^{1,0}\mathbb{S}^6 \oplus T^{0,1}\mathbb{S}^6.$$
%Since $\overline{D}$ preserves $J$, its extension to $T\mathbb{S}^6 \otimes \mathbb{C}$ by $\mathbb{C}$-linearity, followed by restriction to $T^{1,0}\mathbb{S}^6$, is a well-defined connection (still denoted $\overline{D}$). \\

\indent \indent Let $u \colon \Sigma^2 \to \mathbb{S}^6$ be a holomorphic curve in the round $6$-sphere, where $\mathbb{S}^6$ carries its standard strict nearly-K\"{a}hler structure discussed above.  Recalling the connection $\overline{D}$ on $T^{1,0}\mathbb{S}^6 \to \mathbb{S}^6$, we use the same symbol to denote its pullback to $u^*(T^{1,0}\mathbb{S}^6) \to \Sigma$.  % Throughout this work, we will routinely abuse notation by using the same symbol for a connection on a vector bundle $E \to \mathbb{S}^6$ and its pullback connection on $u^*E \to \Sigma$.  In particular, we let $\overline{D}$ denote the pullback connection on $u^*(T^{1,0}\mathbb{S}^6) \to \Sigma$. \\
 Since $\Sigma$ is a Riemann surface, we may endow $u^*(T^{1,0}\mathbb{S}^6) \to \Sigma$ with the Koszul-Malgrange holomorphic structure for $\overline{D}$.  Since $u$ is an immersion, the complex line subbundle $T^{1,0}\Sigma \hookrightarrow u^*(T^{1,0}\mathbb{S}^6)$ is then a holomorphic line subbundle, and the quotient bundle $Q := u^*(T^{1,0}\mathbb{S}^6)/T^{1,0}\Sigma$ inherits a natural holomorphic structure as well. \\ %For $v \in u^*(T^{1,0}\mathbb{S}^6)$, we let $(v) \in Q$ denote its projection to the quotient.  \\%hence on $T^{1,0}\Sigma$ and on $Q := u^*(T^{1,0}\mathbb{S}^6)/T^{1,0}\Sigma$. \\

\indent For computations, let $(e_1, \ldots, e_6)$ be a local $\text{SU}(3)$-frame for $T\mathbb{S}^6$.  Concretely, this means that $(e_1, \ldots, e_6)$ is an oriented orthonormal frame, that
\begin{align*}
Je_1 & = e_2 & Je_3 & = e_4 & Je_5 & = e_6
\end{align*}
and that
$$\Upsilon = (\omega_1 + i\omega_2) \wedge (\omega_3 + i\omega_4) \wedge (\omega_5 + i\omega_6),$$
where $(\omega_1, \ldots, \omega_6)$ is the dual coframe to $(e_1, \ldots, e_6)$.  We also let
\begin{align*}
f_1 & = \frac{1}{2}(e_1 - ie_2) & f_2 & = \frac{1}{2}(e_3 - ie_4) & f_3 & = \frac{1}{2}(e_5 - ie_6),
%\overline{f}_1 & = \frac{1}{2}(e_1 + ie_2) & \overline{f}_2 & = \frac{1}{2}(e_3 + ie_4) & \overline{f}_3 & = \frac{1}{2}(e_5 + ie_6)
\end{align*}
so that $(f_1, f_2, f_3)$ is a local complex $\text{SU}(3)$-frame for $T^{1,0}\mathbb{S}^6$.  Finally, note that
\begin{align*}
\zeta_1 & = \omega_1 + i\omega_2 & \zeta_2 & = \omega_3 + i\omega_4 & \zeta_3 & = \omega_5 + i\omega_6
\end{align*}
is the local $\text{SU}(3)$-frame of $\Lambda^{1,0}(\mathbb{S}^6)$ dual to $(f_1, f_2, f_3)$.

%\indent $(\zeta_1, \zeta_2, \zeta_3)$ \\
% \indent $(f_1, f_2, f_3)$ \\

\indent We now adapt frames to the holomorphic curve $u \colon \Sigma \to \mathbb{S}^6$.  Say that a local $\text{SU}(3)$-frame along $u(\Sigma)$ is an \textit{adapted $\text{U}(2)$-frame} if
\begin{align*}
T\Sigma & = \text{span}(e_1,e_2) & N\Sigma & = \text{span}(e_3, e_4, e_5, e_6).
\end{align*}
An equivalent requirement is that $T^{1,0}\Sigma = \text{span}_{\mathbb{C}}(f_1)$.  With respect to a (local) adapted $\text{U}(2)$-frame, the second fundamental form of $u$ can be expressed as
%$$f_1 \in T^{1,0}\Sigma$$
%The second fundamental form can therefore be expressed locally as
\begin{align*}
\text{I\!I}(e_1, e_1) & = \kappa_1 e_3 + \kappa_2 e_4 + \mu_1 e_5 - \mu_2 e_6 \\
\text{I\!I}(e_1, e_2) & = -\kappa_2 e_3 + \kappa_1 e_4 + \mu_2 e_5 + \mu_1 e_6 
%\text{I\!I}(e_2, e_2) & = -\text{I\!I}(e_1, e_1)
\end{align*}
for some (frame-dependent) functions $\kappa = \kappa_1 + i\kappa_2$ and $\mu = \mu_1 + i\mu_2$.  In \cite[Lemma 4.3]{bryant82}, Bryant shows that
\begin{align*}
\Phi & \in \Gamma(\Lambda^{1,0}\Sigma \otimes \Lambda^{1,0}\Sigma \otimes Q) \\
\Phi & = \kappa\zeta_1 \otimes \zeta_1 \otimes (f_2) + \mu\zeta_1 + \zeta_1 \otimes (f_3)
\end{align*}
is a well-defined (frame-independent) holomorphic section, where the projection of a vector $v \in u^*(T^{1,0}\mathbb{S}^6)$ to the quotient bundle is being denoted $(v) \in Q$.  Notice that, as remarked in \cite[Lemma 4.4]{bryant82}, we have that $\Phi = 0$ identically if and only if $u$ is totally-geodesic. \\

\indent With the above structures in place, it is now simple to prove Theorem \ref{thm:Classical}, which we restate for convenience.

\begin{thm} Let $u \colon \Sigma^2 \to B$ be a compact immersed surface in a geodesic ball $B$ of the round $6$-sphere with $u(\partial \Sigma) \subset \partial B$.  If $u$ is a holomorphic curve, and if $u(\Sigma)$ meets $\partial B$ orthogonally, then $u(\Sigma)$ is totally geodesic.
\end{thm}

\begin{proof} Let $B \subset \mathbb{S}^6$ denote a geodesic ball, let $S := \partial B$ denote its boundary sphere, and let $\nu$ denote the outward-pointing unit normal vector field to $S$.  Let $A \colon TS \to TS$ denote the shape operator of $S \subset \mathbb{S}^6$, i.e.:
$$A(X) = \overline{\nabla}_X\nu$$
The crucial point is that geodesic spheres $S \subset \mathbb{S}^6$ are totally umbilic, meaning that $A = c\,\text{Id}$ for some constant $c > 0$ (namely, the principal curvature).

\indent Let $u \colon \Sigma^2 \to B$ be a holomorphic curve in $B$ with $u(\partial \Sigma) \subset S$ and $u(\Sigma)$ meeting $S$ orthogonally.  For $X \in u^*(T\mathbb{S}^6) = T\Sigma \oplus N\Sigma$, we will write $X = X^{T\Sigma} + X^{N\Sigma}$ for its decomposition into tangential and normal parts.  Now, let $(e_1, e_2)$ be a local oriented orthonormal frame for $T\Sigma$ defined on a neighborhood of a (possibly small) arc $C \subset \partial \Sigma$ such that $\nu = e_1$ along $C$.  Complete $(e_1, e_2)$ to a $\text{U}(2)$-adapted frame $(e_1, e_2, \ldots, e_6)$.  Along $C \subset \partial \Sigma$, we now compute   %Let $p \in \partial \Sigma$, and let $(e_1, e_2)$ be an oriented orthonormal frame for $T\Sigma$ defined on a neighborhood $U$ of $p$ such that $\nu = e_1$ along $\partial \Sigma$.  Complete $(e_1, e_2)$ to a $\text{U}(2)$-adapted frame $(e_1, e_2, \ldots, e_6)$.  Along $\partial \Sigma$, the free-boundary condition gives
$$\text{I\!I}(e_1, e_2) = (\overline{\nabla}_{e_2}\nu)^{N\Sigma} = [A(e_2)]^{N\Sigma} = (ce_2)^{N\Sigma} = 0.$$
Consequently, $\Phi = 0$ along $C \subset \partial \Sigma$, so by holomorphicity, $\Phi = 0$ on all of $\Sigma$, so $u$ is totally-geodesic.
\end{proof}

\section{Holomorphic Curves with Lagrangian Boundary}

\indent \indent We now turn to holomorphic curves with Lagrangian boundary in nearly-K\"{a}hler $6$-manifolds.  In $\S$4.1, we derive a second variation of area formula (Theorem \ref{thm:SecondVarMain}).  Separately, in $\S$4.2 we recall a generalization of the Riemann-Roch formula (Theorem \ref{thm:Riemann-Roch}) that is suited to holomorphic curves with boundary in almost-complex manifolds.  Finally, in $\S$4.3, we deduce the lower bound for the Morse index of area stated in Theorem \ref{thm:LagIndexBound}.  For the Morse index of \textit{energy}, the interested reader may consult \cite{lima2017bounds}.  This section is independent of $\S$3, but draws on the notation set in $\S$2.

\subsection{The Second Variation of Area}

% \indent \indent X \\

\indent \indent Let $u \colon \Sigma^2 \to M^n$ be a compact oriented minimal surface in a Riemannian manifold $M$ such that $u(\partial \Sigma) \subset L$ for some fixed submanifold $L \subset M$.  An \textit{admissible variation vector field} is a normal vector field $\eta \in \Gamma(N\Sigma)$ such that $\eta$ is tangent to $L$ along $\partial \Sigma$.  Let $\nu \in \Gamma(\left.T\Sigma\right|_{\partial \Sigma})$ denote the outer unit conormal, i.e., the unit vector field that is orthogonal to $\partial \Sigma$ and outward-pointing. %--- and let $T \in \Gamma(T\Sigma|_{\partial \Sigma})$ denote the positively-oriented unit vector field that is tangent to $\partial \Sigma$.

\indent It is well-known \cite{fraser1998free} that the second variation of area of $u$ at an admissible variation vector field $\eta \in \Gamma(N\Sigma)$ is
\begin{equation}
(\delta^2A)(\eta) = \int_\Sigma \Vert \nabla^\perp \eta \Vert^2 - \Vert W\eta \Vert^2 - \overline{R}(\eta, e_i, e_i, \eta) + \int_{\partial \Sigma} \langle \overline{\nabla}_\eta \eta, \nu \rangle
\label{eq:SecondVar0}
\end{equation}
%where $\nu$ is the outer unit conormal (i.e., the unit vector field on $\left.T\Sigma\right|_{\partial \Sigma}$ that is orthogonal to $T(\partial \Sigma)$ and outward-pointing),
where $(e_1, e_2)$ is any oriented orthonormal frame for $T\Sigma$.  Here, we recall our conventions from $\S$2.3 that $\overline{\nabla}$ is the ambient covariant derivative, that $\nabla^\top, \nabla^\perp$ are the induced tangential and normal connections, and the corresponding curvature tensors are denoted $\overline{R}, R^\top, R^\perp$, respectively. \\

% \pagebreak
%that $\overline{\nabla}, \nabla^\top, \nabla^\perp$ are the  $M$, that $\nabla^\perp$  and where we recall that $\overline{\nabla}, \nabla^\perp$ are the ambient connection and normal connection, respectively,  the shape operator $W$ from (tag)

\indent We now specialize to the case where $u \colon \Sigma^2 \to M^6$ is a compact holomorphic curve in a nearly-K\"{a}hler $6$-manifold $M^6$ having boundary $u(\partial \Sigma) \subset L$ for some fixed submanifold $L \subset M$.  For the moment, we make no assumptions about $L$.  To begin, we consider the zeroth-order terms $\Vert W\eta \Vert^2$ and $\overline{R}(\eta, e_i, e_i, \eta)$.  We calculate:

\begin{lem} \label{lem:L1} We have
$$\Vert W\eta \Vert^2 + \overline{R}(\eta, e_i, e_i, \eta) = -R^\perp(e_1, e_2, \eta, J\eta) + 2\lambda^2 \Vert \eta \Vert^2.$$
\end{lem}

\begin{proof} First, the nearly-K\"{a}hler curvature identity (\ref{eq:KeyCurvId}) and the constant type equation (\ref{eq:ConstantType}) give
$$\overline{R}(\eta, e_i, e_i, \eta) + \overline{R}(e_1, e_2, \eta, J\eta) = 2 \Vert P(e_1, \eta) \Vert^2 = 2 \lambda^2 \Vert \eta \Vert^2.$$
Now, the Ricci equation (\ref{eq:RicciEq}), followed by the shape operator symmetries (\ref{eq:SymShapeOp}), shows that
\begin{align*}
\overline{R}(e_1, e_2, \eta, J\eta)  & = R^\perp(e_1, e_2, \eta, J\eta) + \langle W_{e_1}\eta, W_{e_2}(J\eta) \rangle - \langle W_{e_1}(J\eta), W_{e_2}\eta \rangle \\
& = R^\perp(e_1, e_2, \eta, J\eta) + \langle W_{e_1}\eta, W_{e_1}\eta \rangle + \langle W_{e_2}\eta, W_{e_2}\eta \rangle \\
& = R^\perp(e_1, e_2, \eta, J\eta) + \Vert W\eta \Vert^2.
\end{align*}
\end{proof}

\indent In view of Lemma \ref{lem:L1}, our second variation formula (\ref{eq:SecondVar0}) now reads:
\begin{equation}
(\delta^2A)(\eta) = \int_\Sigma \Vert \nabla^\perp \eta \Vert^2 + R^\perp(e_1, e_2, \eta, J\eta) - 2\lambda^2 \Vert \eta \Vert^2 + \int_{\partial \Sigma} \langle \overline{\nabla}_\eta \eta, \nu \rangle.
\label{eq:SecondVarV1}
\end{equation}
To handle the first two terms, we require a few pieces of notation.  First, we recall the tensor $P \colon T\Sigma \times N\Sigma \to N\Sigma$ given by
\begin{align*}
%P \colon T\Sigma \times N\Sigma & \to N\Sigma \\
P(X,\eta) & = (\nabla^\perp_XJ)(\eta) = (\overline{\nabla}_XJ)(\eta).
\end{align*}
Second, for a fixed $\eta \in \Gamma(N\Sigma)$, we let $\alpha_\eta \in \Omega^1(\Sigma)$ denote the $1$-form on $\Sigma$ given by
%\noindent \textbf{Notation:} Let $\alpha_\eta \in \Omega^1(\Sigma)$ denote
$$\alpha_\eta(X) = \langle \nabla^\perp_X \eta, J\eta \rangle.$$
Finally, we let $\mathscr{D} \colon \Gamma(N\Sigma) \to \Omega^1(\Sigma) \otimes \Gamma(N\Sigma)$ denote the differential operator
%\noindent \textbf{Notation:} Let
\begin{align}
\mathscr{D}_X\eta & = \nabla^\perp_X\eta + J\nabla^\perp_{JX}\eta.
\label{eq:D-Op}
\end{align}
We now calculate:

\begin{lem} \label{lem:L2}
We have
$$\Vert \nabla^\perp \eta \Vert^2 + R^\perp(e_1, e_2, \eta, J\eta) = \frac{1}{2} \Vert \mathscr{D}\eta \Vert^2 + \langle P(e_1, J\eta), \mathscr{D}_{e_1}\eta \rangle  + d\alpha_\eta(e_1, e_2).$$
\end{lem}

\begin{proof} First, from
$$\frac{1}{2}\Vert \mathscr{D}\eta \Vert^2 =  \Vert \mathscr{D}_{e_1}\eta\Vert^2 = \Vert \nabla^\perp_{e_1}\eta \Vert^2 + \Vert \nabla^\perp_{e_2}\eta \Vert^2 + 2\langle \nabla^\perp_{e_1}\eta, J\nabla^\perp_{e_2}\eta \rangle$$
we observe that
\begin{equation}
\Vert \nabla^\perp \eta \Vert^2 = \frac{1}{2}\Vert \mathscr{D} \eta \Vert^2 -  2\langle \nabla^\perp_{e_1}\eta, J\nabla^\perp_{e_2}\eta \rangle. \label{eq:A}
\end{equation}
Second, we calculate
\begin{align*}
e_1(\alpha_\eta(e_2)) & = \langle \nabla^\perp_{e_1}\nabla^\perp_{e_2}\eta, J\eta \rangle + \langle \nabla^\perp_{e_2}\eta, P(e_1, \eta) \rangle - \langle \nabla^\perp_{e_1}\eta, J\nabla^\perp_{e_2}\eta \rangle \\
e_2(\alpha_\eta(e_1)) & = \langle \nabla^\perp_{e_2}\nabla^\perp_{e_1}\eta, J\eta \rangle + \langle \nabla^\perp_{e_1}\eta, P(e_2, \eta) \rangle + \langle \nabla^\perp_{e_1}\eta, J\nabla^\perp_{e_2}\eta \rangle
\end{align*}
and hence
\begin{align*} 
d\alpha_\eta(e_1, e_2) & = e_1(\alpha_\eta(e_2)) - e_2(\alpha_\eta(e_1)) - \alpha_\eta([e_1, e_2]) \notag \\
& = \langle (\nabla^\perp_{e_1}\nabla^\perp_{e_2} - \nabla^\perp_{e_2}\nabla^\perp_{e_1})\eta, J\eta \rangle - \alpha_\eta([e_1, e_2]) - 2\langle \nabla^\perp_{e_1}\eta, J\nabla^\perp_{e_2} \eta \rangle \notag \\
& \ \ \ \ \ \ \ \ \ \ \ \ \ \ + \langle \nabla^\perp_{Je_1}\eta, P(e_1, \eta) \rangle - \langle \nabla^\perp_{e_1}\eta, P(Je_1, \eta) \rangle  \notag \\
& =  R^\perp(e_1, e_2, \eta, J\eta) - 2\langle \nabla^\perp_{e_1}\eta, J\nabla^\perp_{e_2} \eta \rangle + \langle \nabla^\perp_{Je_1}\eta - J\nabla^\perp_{e_1}\eta, P(e_1, \eta) \rangle \notag \\
& =  R^\perp(e_1, e_2, \eta, J\eta) - 2\langle \nabla^\perp_{e_1}\eta, J\nabla^\perp_{e_2} \eta \rangle - \langle \mathscr{D}_{e_1}\eta, P(e_1, J\eta) \rangle %  \label{eq:B}
\end{align*}
so that
\begin{equation}
- 2\langle \nabla^\perp_{e_1}\eta, J\nabla^\perp_{e_2} \eta \rangle = -R^\perp(e_1, e_2, \eta, J\eta) + \langle P(e_1, J\eta), \mathscr{D}_{e_1}\eta \rangle + d\alpha_\eta(e_1, e_2). \label{eq:B} 
\end{equation}
From (\ref{eq:A}) and (\ref{eq:B}), we obtain the result.
%\begin{align*}
%\Vert D^\perp \eta \Vert^2 & = 2\Vert \mathscr{D} \eta \Vert^2 + \langle P(e_1, J\eta), \mathscr{D}_{e_1}\eta \rangle - R^\perp(e_1, e_2, \eta, J\eta) + d\alpha_\eta(e_1, e_2).
%\end{align*}
\end{proof}

\begin{rmk} The formula of Lemma \ref{lem:L2} implies that the quantity $\langle P(e_1, J\eta), \mathscr{D}_{e_1}\eta \rangle$ is independent of the choice of orthonormal frame $(e_1, e_2)$.  Here is another way to see this.  For $\eta \in \Gamma(N\Sigma)$, consider the tensor $G_\eta \colon T\Sigma \otimes T\Sigma \to \mathbb{R}$ via
$$G_\eta(X,Y) = \langle P(X,J\eta), \mathscr{D}_Y\eta\rangle = \textstyle \frac{1}{3}\,d\omega(X, J\eta, \mathscr{D}_Y\eta).$$
The symmetries of $P$ and $\mathscr{D}$ show that $G_\eta(X,Y) = G_\eta(JX,JY)$, so $G_\eta(e_1, e_1) = G_\eta(e_2, e_2)$.  It follows that $\langle P(e_1, J\eta), \mathscr{D}_{e_1}\eta\rangle = \frac{1}{2}\,\text{tr}(G_\eta)$, which is clearly independent of $(e_1, e_2)$.
\end{rmk}

\indent Using Lemma \ref{lem:L2} in (\ref{eq:SecondVarV1}), followed by Stokes' Theorem, our second variation formula is
\begin{equation}
(\delta^2A)(\eta) = \int_\Sigma \frac{1}{2} \Vert \mathscr{D}\eta \Vert^2 + \langle P(e_1, J\eta), \mathscr{D}_{e_1}\eta \rangle  - 2\lambda^2 \Vert \eta \Vert^2 + \int_{\partial \Sigma} \langle \overline{\nabla}_\eta \eta, \nu \rangle + \langle \nabla^\perp_T\eta, J\eta \rangle
\label{eq:SecondVarV2}
\end{equation}
where $T \in \Gamma(T\Sigma|_{\partial \Sigma})$ is the positively-oriented unit vector field tangent to $\partial \Sigma$.  Thus far, we have not imposed any conditions on the submanifold $L$.  We now suppose that $L$ is Lagrangian and again exploit the fact that $u$ is holomorphic and $M$ is nearly-K\"{a}hler:

\begin{lem} \label{lem:L3} If $L$ is Lagrangian, then
$$\langle \overline{\nabla}_\eta \eta, \nu \rangle + \langle \nabla^\perp_T\eta, J\eta \rangle = 0.$$
\end{lem}

\begin{proof} Along $\partial \Sigma$, we have that $T$ and $\eta$ are tangent to $L$, so that $[T,\eta]$ must be tangent to $L$.  On the other hand, since $L$ is Lagrangian, $J\eta$ is normal to $L$.  Therefore,
\begin{equation}
\langle \overline{\nabla}_T\eta - \overline{\nabla}_\eta T, J\eta \rangle = \langle [T,\eta], J\eta \rangle = 0.
\label{eq:Tangency}
\end{equation}
Consequently,
\begin{align*}
\langle \nabla^\perp_T\eta, J\eta \rangle = \langle \overline{\nabla}_T\eta, J\eta \rangle = \langle \overline{\nabla}_\eta T, J\eta \rangle = -\langle T, \overline{\nabla}_\eta(J\eta) \rangle = -\langle T, J\overline{\nabla}_\eta\eta \rangle & = \langle JT, \overline{\nabla}_\eta \eta \rangle \\
& = -\langle \nu, \nabla_\eta \eta \rangle,
\end{align*}
where the second equality uses (\ref{eq:Tangency}), the third equality uses $\eta(\langle T, J\eta \rangle) = 0$, the fourth equality uses the nearly-K\"{a}hler condition $(\overline{\nabla}_XJ)(X) = 0$, and the last equality uses that the holomorphicity of $u$ implies $JT = -\nu$.
\end{proof}

\indent From Lemma \ref{lem:L3} and (\ref{eq:SecondVarV2}), and recalling that $\langle P(X,Y),Z \rangle = (\overline{\nabla}_X\omega)(Y,Z) = \frac{1}{3}d\omega(X,Y,Z)$ for all $X,Y,Z \in TM$, we arrive at our desired formula:

\begin{thm} \label{thm:SecondVarMain} Let $u \colon \Sigma^2 \to M^6$ be a compact holomorphic curve in a nearly-K\"{a}hler $6$-manifold with boundary $u(\partial \Sigma) \subset L$ for a Lagrangian submanifold $L \subset M$.  For any normal vector field $\eta \in \Gamma(N\Sigma)$ that is tangent to $L$ along $\partial \Sigma$, the second variation of area is
$$(\delta^2A)(\eta) = \int_\Sigma \frac{1}{2} \Vert \mathscr{D}\eta \Vert^2 + \frac{1}{3} d\omega(e, J\eta, \mathscr{D}_{e}\eta)  - 2\lambda^2 \Vert \eta \Vert^2$$
where $\mathscr{D}$ is the operator defined in (\ref{eq:D-Op}), where $\lambda \geq 0$ is the type constant (\ref{eq:ConstantType}), and where $e \in \Gamma(T\Sigma)$ is any unit tangent vector field.
\end{thm}

\subsection{The Riemann-Roch Theorem}

\indent \indent From Theorem \ref{thm:SecondVarMain}, we see that any admissible vector field $\eta \in \Gamma(N\Sigma)$ with $\mathscr{D}\eta = 0$ will satisfy $(\delta^2A)(\eta) \leq 0$, with equality only if $\lambda = 0$ or $\eta = 0$.  Therefore, if $\lambda \neq 0$, then the Morse index of $u$ is at least % $\dim_{\mathbb{R}}(\text{Ker}(\mathscr{D}))$, which in turn is at least
the Fredholm index of $\mathscr{D}$ on admissible variations.  This latter quantity can be computed by means of a well-known index formula for real Cauchy-Riemann operators.  Recalling this formula (Theorem \ref{thm:Riemann-Roch}) is the aim of this section.  The material here is by now standard; our rather brief discussion is drawn from \cite[Appendix C]{mcduff2012j} and \cite[$\S$2.3, $\S$3.4]{wendl2010lectures}.  % The reader seeking more details may wish to consult those texts.

% The material in this section is standard; our discussion here follows those in \cite[Appendix C]{mcduff2012j} and \cite[$\S$2.3, $\S$3.4]{wendl2010lectures}.  The reader seeking more details is advised to consult those references.

%Our discussion follows those in [McD-Sal, Appendix C] and [Wendl, $\S$2.3, $\S$3.4]; the reader seeking more information is advised to consult those sources.

\subsubsection{Connections on Complex Vector Bundles}

\indent \indent Let $E \to \Sigma$ be a real vector bundle over a real manifold $\Sigma$.  Recall that a \textit{(real) connection} on $E$ is an $\mathbb{R}$-linear operator
$$\nabla \colon \Gamma(E) \to \Gamma(\text{Hom}_{\mathbb{R}}(T\Sigma, E)) = \Gamma(\Lambda^1(\Sigma; \mathbb{R}) \otimes_{\mathbb{R}} E)$$
%$$\nabla \colon \Gamma(E) \to \Gamma(\Lambda^1(\Sigma) \otimes_{\mathbb{R}} E)$$
that satisfies the Leibniz rule
\begin{equation}
\nabla(f\xi) = f\nabla \xi + df \otimes \xi \label{eq:Leibniz}
\end{equation}
for all $f \in C^\infty(\Sigma; \mathbb{R})$ and all $\xi \in \Gamma(E)$.  From now on, let us suppose that $E \to \Sigma$ has the structure of a complex bundle, say $J \colon E \to E$.  A real connection $\nabla$ is said to \textit{preserve $J$} if $\nabla_X(J\xi) = J\nabla_X \xi$ for all $X \in T\Sigma$ and $\xi \in \Gamma(E)$.

\indent Let $\text{Hom}_{\mathbb{C}}(T\Sigma^{\mathbb{C}}, E) \to \Sigma$ denote the vector bundle whose fiber at $p \in \Sigma$ consists of the $(i,J)$-linear maps $T_p\Sigma^{\mathbb{C}} \to E_p$.  Each fiber carries an obvious $\mathbb{C}$-module structure via $i \cdot \alpha := J\circ \alpha = \alpha \circ i$.  With this understood, we recall that a \textit{complex connection} on $(E,J)$ is a $\mathbb{C}$-linear operator
$$\nabla \colon \Gamma(E) \to \Gamma(\text{Hom}_{\mathbb{C}}(T\Sigma^{\mathbb{C}}, E)) = \Gamma(\Lambda^1(\Sigma; \mathbb{C}) \otimes_{\mathbb{C}} E)$$
that satisfies the Leibniz rule (\ref{eq:Leibniz}) for all $f \in C^\infty(\Sigma; \mathbb{C})$ and $\xi \in \Gamma(E)$.  Of course, there is a natural bijection between complex connections and $J$-preserving real connections. \\

\indent We now suppose that $(\Sigma, j)$ is a Riemann surface.  Let $\text{Hom}^+(T\Sigma, E)$ and $\text{Hom}^-(T\Sigma; E)$ denote the vector bundles over $\Sigma$ whose fiber at $p \in \Sigma$ consists of the $(j,J)$-linear and $(j,J)$-antilinear maps $T_p\Sigma \to E_p$, respectively.  Explicitly,
\begin{align*}
\text{Hom}^+(T\Sigma, E)|_p & = \{A \in \text{Hom}_{\mathbb{R}}(T_p\Sigma, E_p) \colon J \circ A = A \circ j\} \\
\text{Hom}^-(T\Sigma, E)|_p & = \{A \in \text{Hom}_{\mathbb{R}}(T_p\Sigma, E_p) \colon J \circ A = -A \circ j\}.
\end{align*}
Given a real connection $\nabla$ on $(E,J) \to (\Sigma,j)$, we define the operators
\begin{align*}
\partial^\nabla \colon \Gamma(E) & \to \Gamma( \text{Hom}^+(T\Sigma, E)) & \overline{\partial}^\nabla \colon \Gamma(E) & \to \Gamma( \text{Hom}^-(T\Sigma, E)) \\
\partial^\nabla\xi & := \frac{1}{2}\left(\nabla \xi - J \circ \nabla \xi \circ j \right) & \overline{\partial}^\nabla \xi & := \frac{1}{2}\left(\nabla \xi + J \circ \nabla \xi \circ j\right)
\end{align*}
In the case of the trivial complex line bundle $(\underline{\mathbb{C}},i) \to \Sigma$, the exterior derivative $d \colon C^\infty(\Sigma; \mathbb{C}) \to  \Omega^1(\Sigma; \mathbb{C}) = \Gamma( \Lambda^1(\Sigma; \mathbb{R}) \otimes_{\mathbb{R}} \underline{\mathbb{C}}) = \Gamma(\text{Hom}_{\mathbb{R}}(T\Sigma, \underline{\mathbb{C}}))$ is a real connection that preserves $i$.  The corresponding operators are denoted
\begin{align*}
\partial \colon C^\infty(\Sigma; \mathbb{C}) & \to \Gamma( \text{Hom}^+(T\Sigma, \underline{\mathbb{C}})) & \overline{\partial} \colon C^\infty(\Sigma; \mathbb{C}) & \to \Gamma( \text{Hom}^-(T\Sigma, \underline{\mathbb{C}})) \\
\partial f & := \frac{1}{2}\left(df - i\,df \circ j \right) & \overline{\partial}f & := \frac{1}{2}\left(df + i\,df \circ j\right)
\end{align*}
\noindent Finally, we remark that $\text{Hom}^+(T\Sigma, \underline{\mathbb{C}}) \cong \Lambda^{1,0}(\Sigma)$ and $\text{Hom}^-(T\Sigma, \underline{\mathbb{C}}) \cong \Lambda^{0,1}(\Sigma)$.  More generally,
\begin{align*}
\text{Hom}^+(T\Sigma, E) & \cong \text{Hom}_{\mathbb{C}}(T^{1,0}\Sigma, E^{1,0}) \cong \Lambda^{1,0}(\Sigma) \otimes_{\mathbb{C}} E \\
\text{Hom}^-(T\Sigma, E) & \cong \text{Hom}_{\mathbb{C}}(T^{0,1}\Sigma, E^{1,0}) \cong \Lambda^{0,1}(\Sigma)\otimes_{\mathbb{C}} E.
\end{align*}

\subsubsection{Cauchy-Riemann Operators}

\indent \indent Given a complex bundle $(E,J) \to \Sigma$ over a Riemann surface $(\Sigma,j)$, the operators of the form $\overline{\partial}^\nabla$ for some complex connection $\nabla$ are particularly important.  They are called \textit{complex Cauchy-Riemann operators} (or \textit{$\overline{\partial}$-operators} or \textit{holomorphic structures}), and their Fredholm indices are the subject of the classical Riemann-Roch Theorem.  To be precise:

\begin{defn} An $\mathbb{R}$-linear operator $\mathcal{D}_0 \colon \Gamma(E) \to \Gamma(\text{Hom}^-(T\Sigma, E))$ is called a \textit{complex Cauchy-Riemann operator} if either of the following equivalent conditions hold: \\
\indent (i) For all $f \in C^\infty(\Sigma; \mathbb{R})$ and all $\xi \in \Gamma(E)$:
\begin{align*}
\mathcal{D}_0(f\xi) & = f\,\mathcal{D}_0\xi + \overline{\partial}f \otimes \xi & \mathcal{D}_0(J\xi) & = J\,\mathcal{D}_0\xi.
\end{align*}
\indent (ii) There exists a $J$-preserving real connection $\nabla$ on $E$ for which $\mathcal{D}_0 = \overline{\partial}^\nabla$.
\end{defn}

\indent For certain geometric applications --- especially those involving holomorphic curves in (non-integrable) almost-complex manifolds --- one often needs to consider the wider class of $\overline{\partial}^\nabla$ operators for which $\nabla$ does not necessarily preserve the complex structure $J$.  Such operators are called:

\begin{defn} An $\mathbb{R}$-linear operator $\mathcal{D} \colon \Gamma(E) \to \Gamma(\text{Hom}^-(T\Sigma, E))$ is called a \textit{real Cauchy-Riemann operator} if any of the following equivalent conditions hold: \\
\indent (i) For all $f \in C^\infty(\Sigma; \mathbb{R})$ and $\xi \in \Gamma(E)$:
$$\mathcal{D}(f\xi) = f\,\mathcal{D}\xi + \overline{\partial}f \otimes \xi.$$
\indent (ii) There exists a real connection $\nabla$ on $E$ for which $\mathcal{D} = \overline{\partial}^\nabla$. \\
\indent (iii) There exist a section $\alpha \in \Gamma(E^* \otimes_{\mathbb{R}} \text{Hom}^-(T\Sigma, E))$ and a complex Cauchy-Riemann operator $\mathcal{D}_0$ on $E$ for which $\mathcal{D} = \mathcal{D}_0 + \alpha$.
\end{defn} 

\indent Finally, recall that a subbundle $F \subset E|_{\partial \Sigma}$ is \textit{totally real} if each fiber is a totally real subspace, meaning that $J(F_p) \cap F_p = \{0\}$ for all $p \in \partial \Sigma$.  In this case, we define
\begin{equation}
W^{k,p}_F(E) := \{\xi \in W^{k,p}(E) \colon \xi(\partial \Sigma) \subset F\}. \label{eq:TotRealBdry}
\end{equation}
Now, on a complex vector bundle $(E,J) \to (\Sigma, j)$ equipped with a totally real subbundle $F \subset E|_{\partial \Sigma}$ and a Cauchy-Riemann operator $\mathcal{D}$, we let $\mathcal{D}_F \colon W^{k,p}_F(E) \to W^{k-1,p}(\text{Hom}^-(T\Sigma, E))$ denote the restriction of $\mathcal{D}$ to $W^{k,p}_F(E) \subset W^{k,p}(E)$.  It turns out that the boundary condition (\ref{eq:TotRealBdry}) is Fredholm, and the Fredholm index of $\mathcal{D}_F$ can be calculated via the following Riemann-Roch formula:
%Given a real Cauchy-Riemann operator $\mathcal{D}$ on a complex vector bundle $(E,J) \to \Sigma$, and given a totally real subbundle $F \subset E|_{\partial \Sigma}$, 
%\indent The Fredholm index of real Cauchy-Riemann operators can be calculated via the following fundamental formula:

\begin{thm}[\cite{mcduff2012j}] \label{thm:Riemann-Roch} Let $E \to \Sigma$ be a complex vector bundle of complex rank $r$ over a compact Riemann surface $\Sigma$ with boundary, and let $F \subset E|_{\partial \Sigma}$ be a totally real subbundle.  Let $\mathcal{D}$ denote a real Cauchy-Riemann operator on $E$ of class $W^{\ell-1,p}$, where $\ell \in \mathbb{Z}^+$ and $p > 1$ satisfy $\ell p > 2$.  Then the real Fredholm index of $\mathcal{D}_F$ is
$$\mathrm{Ind}(\mathcal{D}_F) = r\chi(\Sigma) + \mu(E,F),$$
where $\chi(\Sigma)$ is the Euler characteristic and $\mu(E,F)$ is the boundary Maslov index.
\end{thm}

\subsection{The Morse Index}

\indent \indent We now return to the study of holomorphic curves $u \colon \Sigma^2 \to M^6$ with boundary in a Lagrangian $L \subset M$ of a nearly-K\"{a}hler $6$-manifold $M$.  Recall Theorem \ref{thm:SecondVarMain}: For a normal vector field $\eta \in \Gamma(N\Sigma)$ that is tangent to $L$ along $\partial \Sigma$, we have
\begin{equation}
(\delta^2A)(\eta) = \int_\Sigma \frac{1}{2} \Vert \mathscr{D}\eta \Vert^2 + \frac{1}{3} d\omega(e, J\eta, \mathscr{D}_{e}\eta)  - 2\lambda^2 \Vert \eta \Vert^2 \label{eq:SecondVarRef}
\end{equation}
The operator $\mathscr{D} \colon \Gamma(N\Sigma) \to \Gamma(\text{Hom}^-(T\Sigma, N\Sigma))$ appearing in this formula, defined in (\ref{eq:D-Op}), is a real Cauchy-Riemann operator.  The proof of Theorem \ref{thm:LagIndexBound} now follows quickly:

%\indent.We now observe that the operator $\mathscr{D} \colon \Gamma(N\Sigma) \to \Gamma(\text{Hom}^-(T\Sigma, N\Sigma))$ defined in (\ref{eq:D-Op}) is a real Cauchy-Riemann operator, so $\ldots$  The proof of Theorem \ref{thm:LagIndexBound} now follows quickly:

\begin{proof} Suppose $M^6$ is a \textit{strict} nearly-K\"{a}hler $6$-manifold, so that $\lambda \neq 0$.  Let us orthogonally decompose $TL|_{\partial \Sigma} = T(\partial \Sigma) \oplus F$.  In other words, $F = \left.TL\right|_{\partial \Sigma} \cap \left.N\Sigma\right|_{\partial \Sigma}$, and hence is a totally real subbundle of $N\Sigma|_{\partial \Sigma}$.  In this language, the set admissible normal vector fields is
%Let $F \subset N\Sigma|_{\partial \Sigma}$ be the totally real subbundle given by $F := \left.TL\right|_{\partial \Sigma} \cap \left.N\Sigma\right|_{\partial \Sigma}$.   that the set of admissible normal vector fields is
$$\mathcal{A} := \{\eta \in \Gamma(N\Sigma) \colon \eta|_{\partial \Sigma} \in TL\} = \{\eta \in \Gamma(N\Sigma) \colon \eta(\partial \Sigma) \subset F\}.$$
Let $\mathcal{H} := \left\{\eta \in \mathcal{A} \colon \mathscr{D}\eta = 0 \right\}$.  If $\eta \in \mathcal{H}$ and $\eta \not\equiv 0$, then (\ref{eq:SecondVarRef}) shows that %Theorem \ref{thm:SecondVarMain} shows that
$$(\delta^2A)(\eta) = -2\lambda^2 \int_\Sigma \Vert \eta \Vert^2 < 0.$$
Thus, the Riemann-Roch Theorem \ref{thm:Riemann-Roch} implies that the Morse index of $u$ satisfies
\begin{align*}
\mathrm{Ind}(u) \geq \dim_{\mathbb{R}}(\mathcal{H}) \geq \text{Ind}(\mathscr{D}_F) & = 2\chi(\Sigma) + \mu(N\Sigma, F) \\
& = \mu(T\Sigma, T(\partial \Sigma)) + \mu(N\Sigma, F) \\
& = \mu(u^*(TM), TL).
\end{align*}
Here, we used that $\mu(T\Sigma, T(\partial \Sigma)) = 2\chi(\Sigma)$ for any compact surface with boundary \cite[$\S$5]{pacini2019maslov}. In the last equality, we used the splittings $u^*(TM) = T\Sigma \oplus N\Sigma$ and $TL|_{\partial \Sigma} = T(\partial \Sigma) \oplus F$ together with the additivity \cite[Appendix C.3]{mcduff2012j} of the boundary Maslov index.
%In other words,
%$$\mathcal{H} \subset \{\eta \in \Gamma(N\Sigma) \colon (\delta^2A)(\eta) < 0 \}$$
\end{proof}

%\section{Proofs of Main Theorems}
%
%\subsection{The Round $6$-Sphere}
%
%X 
%
%\subsection{Holomorphic Curves in the $6$-Sphere}
%
%X
%
%
%\subsection{The Classical Free-Boundary Problem}
%
%\indent \indent Adapt frames. \\
%
%\indent X \\
%
%
%\subsection{The Lagrangian Boundary Problem}
%
%X
%
%\begin{thm}[Riemann-Roch] X
%\end{thm}
%
%X

% \pagebreak
\bibliographystyle{plain}
\bibliography{FBRef}
\Addresses

\end{document}